\documentclass[11pt]{amsart}

\usepackage[T1]{fontenc}
\usepackage[latin1]{inputenc}
\usepackage[english]{babel}
\usepackage{amsmath,amssymb,amsthm}
\usepackage{enumitem}
\usepackage{mymacros}
\usepackage[margin=1.5in]{geometry}
\usepackage{hyperref}
\renewcommand{\MR}[1]{}

\newcommand{\Han}{\operatorname{Han}}
\newcommand{\HAN}{\operatorname{HAN}}
\newcommand{\mcc}{M\textsuperscript{c}Carthy}
\newcommand{\CB}{\operatorname{CB}}

\title[Multipliers of weak products]{Multipliers and operator space structure of weak product spaces}
\author{Rapha\"el Clou\^atre}
\address{Department of Mathematics, University of Manitoba, Winnipeg, Manitoba, Canada R3T 2N2}
\thanks{R.C.\ was partially supported by an NSERC Discovery grant.}
\email{raphael.clouatre@umanitoba.ca}
\author{Michael Hartz}
\address{Fakult\"at f\"ur Mathematik und Informatik, FernUniversit\"at in Hagen, 58084 Hagen, Germany}
\curraddr{Fachrichtung Mathematik, Universit\"at des Saarlandes, 66123 Saarbr\"ucken, Germany}
\email{hartz@math.uni-sb.de}
\thanks{A large part of the research for this article was carried out while R.C.\ visited M.H.\
  at Fern\-Universit\"at in Hagen. The financial support provided by FernUniversit\"at in Hagen
is gratefully acknowledged.}


\begin{document}

\begin{abstract}
  In the theory of reproducing kernel Hilbert spaces, weak product spaces generalize
  the notion of the Hardy space $H^1$. For complete Nevanlinna--Pick spaces $\mathcal H$, we characterize
  all multipliers of the weak product space $\mathcal H \odot \mathcal H$. In particular,
  we show that if $\cH$ has the so-called column-row property, then the multipliers of $\cH$
  and of $\cH \odot \cH$ coincide. This result applies in particular to the classical Dirichlet space
  and to the Drury--Arveson space on a finite dimensional ball.
  As a key device, we exhibit a natural operator space structure on $\cH \odot \cH$, which enables
  the use of dilations of completely bounded maps.
\end{abstract}

\maketitle

\section{Introduction}

Let $\cH$ be a reproducing kernel Hilbert space of functions on a set $X$.
The weak product space is defined by
\begin{equation*}
  \cH \odot \cH =
  \Big\{ h = \sum_{i=1}^\infty f_i g_i : \sum_{i=1}^\infty \|f_i\| \, \|g_i\| < \infty \Big\},
\end{equation*}
with norm
\begin{equation*}
  \|h\|_{\cH \odot \cH} = \inf
  \Big\{ \sum_{i=1}^\infty \|f_i\| \, \|g_i\| : h = \sum_{i=1}^\infty f_i g_i  \Big\}.
\end{equation*}
This is a Banach function space on $X$, meaning in particular that the functionals of evaluation at points in $X$
are continuous on $\cH \odot \cH$. If $\cH = H^2(\bD)$, the Hardy space on the unit disc $\bD$,
then $\cH \odot \cH = H^1(\bD)$ with equality of norms. In fact, every function in $H^1(\bD)$
is a product of two functions in $H^2(\bD)$.
The notion of a weak product space has its origins in a famous paper of Coifman, Rochberg
and Weiss \cite{CRW76}. There, it is shown that the Hardy spaces and Bergman spaces on the unit ball $\bB_d$
satisfy $H^1( \partial \bB_d) = H^2(\partial \bB_d) \odot H^2 (\partial \bB_d)$
and $L^1_a(\bB_d) = L^2_a(\bB_d) \odot L^2_a(\bB_d)$, with equivalence of norms.
In general, one can regard weak product spaces as a replacement for the Hardy space $H^1$ in the context
of an arbitrary reproducing kernel Hilbert space. In this setting,
weak product spaces are closely related to boundedness of Hankel forms; see e.g.\
the introduction of \cite{ARS+10}. Weak product spaces have been concretely
studied for instance for the classical Dirichlet  space \cite{ARS+10,LR15,RS14}, the Drury--Arveson
space \cite{RS16} and more generally for complete Nevanlinna--Pick spaces \cite{AHM+18,JM18}.

If $\cB$ is a Banach function space on $X$, the multiplier algebra of $\cB$ is defined by
\begin{equation*}
  \Mult(\cB) = \{ \varphi: X \to \bC: \varphi \cdot f \in \cB \text{ for all } f \in \cB \}.
\end{equation*}
The closed graph theorem implies that for each $\varphi \in \Mult(\cB)$, the associated
multiplication operator $M_\varphi$ on $\cB$ is bounded, so we may define $\|\varphi\|_{\Mult(\cB)}
= \|M_\varphi\|_{B(\cB)}$. It is immediate from the definition of the weak product space
that $\Mult(\cH) \subset \Mult(\cH \odot \cH)$, with $\|\varphi\|_{\Mult(\cH \odot \cH)} \le
\| \varphi\|_{\Mult(\cH)}$ for all $\varphi \in \Mult(\cH)$.
On the other hand, for the classical Hardy space, it is not hard to see that both
$\Mult(H^1(\bD))$ and $\Mult(H^2(\bD))$ agree with $H^\infty(\bD)$, with equality of norms.
This naturally raises the following question.

\begin{quest}
  \label{quest:mult_same} Is $\Mult(\cH \odot \cH) = \Mult(\cH)$?
\end{quest}

This question was studied by Richter and Wick \cite{RW16}, who provided a positive answer
for first order weighted Besov spaces on the unit ball  using function theoretic estimates.
In particular, their
result applies to the classical Dirichlet space and to the Drury--Arveson space
$H^2_d$ for $d \le 3$, but not to $H^2_d$ for $d \ge 4$.

In this article, we characterize multipliers of $\cH \odot \cH$ for
normalized complete Nevanlinna--Pick spaces
and are thus able to give a positive answer to Question \ref{quest:mult_same} in many instances.
The prototypical example of a normalized complete Nevanlinna--Pick space is the Hardy space $H^2(\bD)$,
but there are many
other examples such as the classical Dirichlet space or the Drury-Arveson space $H^2_d$
for $d \in \bN \cup \{\infty\}$.
Normalized complete Nevanlinna--Pick spaces
can be defined in terms of the validity of a suitable
version of the Nevanlinna--Pick interpolation theorem. Equivalently,
by results of McCullough, Quiggin and Agler--\mcc, a reproducing kernel Hilbert space
is a normalized complete Nevanlinna--Pick space if and only if its reproducing kernel
$K$ is of the form
\begin{equation*}
  K(z,w) = \frac{1}{1 - \langle b(z),b(w) \rangle} \quad (z,w \in X),
\end{equation*}
where $b$ maps $X$ into the open unit ball of an auxiliary Hilbert space and satisfies $b(z_0) = 0$
for some distinguished point $z_0 \in X$. This characterization and more background information
can be found in \cite{AM02}.

A key ingredient in our analysis of $\Mult(\cH \odot \cH)$ is the observation
that $\cH \odot \cH$ can be equipped with a natural operator space structure, and we will use
this additional structure crucially.
Briefly,
$\cH \odot \cH$ can be identified as the dual of the space of compact Hankel operators
on $\cH$ (see \cite[Section 2]{AHM+18}), hence $\cH \odot \cH$ is the dual of an operator
space and thus itself an operator space.
For more details, see Section \ref{sec:prelim}.
In particular, for each $n \in \bN$, the space
$M_n(\cH \odot \cH)$ of $n \times n$ matrices with entries in $\cH \odot \cH$ carries a natural norm.
As is customary in operator space theory, for each $n\in \bN$ one associates with a linear map $A: \cH \odot \cH \to \cH \odot \cH$
its amplification $A^{(n)}: M_n(\cH \odot \cH) \to M_n(\cH \odot \cH)$, defined by applying $A$ entrywise.
The linear map $A$ is then said to be a complete contraction if each map $A^{(n)}$ is a contraction.
In particular, given a function $\theta: X \to \bC$, we say that $\theta$ is a contractive multiplier
of $\cH \odot \cH$
if the multiplication operator $M_\theta: \cH \odot \cH \to \cH \odot \cH$ is a contraction,
and that $\theta$ is a completely contractive multiplier of $\cH \odot \cH$ if $M_\theta$
is a complete contraction.

Let $M^C_1(\cH)$ be the space of sequences of multipliers of $\cH$ that are contractive when viewed as a column operator on $\cH$.
In other words, a sequence $(\varphi_n)$ belongs to $M^C_1(\cH)$ if and only if the operator
\[
 \begin{bmatrix}
     M_{\varphi_1} \\ M_{\varphi_2} \\ \vdots
    \end{bmatrix}: \cH \to \cH \otimes \ell^2
\]
is contractive.
Our main result can now be stated as follows.

\begin{thm}
  \label{thm:mult_char}
  Let $\cH$ be a normalized complete Nevanlinna--Pick space on $X$.
  The following assertions are equivalent for a function $\theta: X \to \bC$.
  \begin{enumerate}[label=\normalfont{(\roman*)}]
    \item The function $\theta$ is a contractive multiplier of $\cH \odot \cH$.
    \item The function $\theta$ is a completely contractive multiplier of $\cH \odot \cH$.
    \item There exist sequences $(\varphi_n)_{n=1}^\infty, (\psi_n)_{n=1}^\infty$ in $M^C_1(\cH)$ such that
      $\theta = \sum_{n=1}^\infty \varphi_n \psi_n$.
  \end{enumerate}
\end{thm}

In particular, it follows that the norm of an element $\theta \in \Mult(\cH \odot \cH)$ is given by
\begin{equation*}
  \|\theta\|_{\Mult(\cH \odot \cH)}
  = \inf \Bigg\{
    \left\|
    \begin{bmatrix}
      \varphi_1 \\ \varphi_2 \\ \vdots 
    \end{bmatrix} \right\|
    \left\|
    \begin{bmatrix}
      \psi_1 \\ \psi_2 \\ \vdots 
  \end{bmatrix} \right\| : \theta = \sum_{n=1}^\infty \varphi_n \psi_n \Bigg\},
\end{equation*}
where the norms of the columns in the infimum are taken in $\Mult(\cH, \cH \otimes \ell^2)$. Moreover, the infimum is attained.

The proof of this theorem will be separated in several steps.
The implication (iii) $\Rightarrow$ (i) easily follows from the definition
of $\cH \odot \cH$ and already appears in \cite[Theorem 3.1]{AHM+18}.
We provide the short argument in Proposition \ref{prop:factorization_is_contractive}.
The proof of the implication (i) $\Rightarrow$ (ii)
uses a recent result of Jury and Martin \cite{JM18} and is presented in Proposition \ref{prop:mult_cc}.
The majority of the work occurs in the proof of the implication (ii) $\Rightarrow$ (iii),
which uses dilation theory and is done in Theorem \ref{thm:multipliers_factor}.

While not logically necessary, we also provide a direct proof of the implication (iii) $\Rightarrow$ (ii);
this is done in Proposition \ref{prop:factorization_multiplier}. This proof shows how the operator
space structure of $\cH \odot \cH$ enters the picture and motivates our approach to the harder
implication (ii) $\Rightarrow$ (iii).

In many instances, Theorem \ref{thm:mult_char} implies an affirmative answer to Question \ref{quest:mult_same}.
The key property is the following.
A normalized complete Nevanlinna--Pick space $\cH$ is said to satisfy the \emph{column-row property (with constant $\kappa$)} if whenever $(\varphi_n)_{n=1}^\infty$ is a sequence in $\Mult(\cH)$ with
\begin{equation*}
      \left\| \begin{bmatrix}
        \varphi_1 \\ \varphi_2 \\ \vdots
      \end{bmatrix} \right\|_{\Mult(\cH, \cH \otimes \ell^2)} \le 1,
\end{equation*}
then
\begin{equation*}
      \Big\| \begin{bmatrix}
        \varphi_1 & \varphi_2 & \cdots
      \end{bmatrix} \Big\|_{\Mult(\cH \otimes \ell^2, \cH)} \le \kappa.
\end{equation*}
The classical Hardy space $H^2(\bD)$ satisfies the column-row property with constant $1$,
as the norm of a row and of a column of multipliers are both given by the supremum norm
over $\bD$.
It is a result of Trent \cite{Trent04}
that the classical Dirichlet space satisfies the column-row property with constant at most $\sqrt{18}$.
The Drury--Arveson space $H^2_d$ with $d < \infty$ also satisfies the column-row property
with some finite constant, which possibly depends on $d$ \cite[Theorem 1.5]{AHM+18}.
More generally, it was shown in \cite{AHM+18a} that radially weighted Besov spaces on the unit ball in finite dimensions
satisfy the column-row property with a finite constant, which again possibly depends on the dimension and on the particular
space.
It is an open question whether every normalized complete Nevanlinna--Pick space satisfies
the column-row property with a finite constant (see \cite{Pascoe19} for some work on this question).
Recently, it has become clear that the column-row property is a very useful technical property
when studying weak product spaces, see for instance \cite{AHM+18}. A striking example
of this is the result of Jury and Martin \cite{JM18}, according to which every function
in $\cH \odot \cH$ factors as a product of precisely two functions in $\cH$, provided that $\cH$ satisfies
the column-row property.

Theorem \ref{thm:mult_char} combined with an argument from  \cite[Theorem 3.1]{AHM+18} implies a positive answer to Question \ref{quest:mult_same} in the
presence of the column-row property.

\begin{cor}
  \label{cor:column_row}
  Let $\cH$ be a normalized complete Nevanlinna--Pick space that satisfies the column row-property
  with constant $\kappa$. Then $\Mult(\cH) = \Mult(\cH \odot \cH)$ and
  \begin{equation*}
    \|\theta\|_{\Mult(\cH \odot \cH)} \le \|\theta\|_{\Mult(\cH)} \le \kappa \| \theta \|_{\Mult(\cH \odot \cH)}
  \end{equation*}
  for every $\theta \in \Mult(\cH)$.
\end{cor}

\begin{proof}
  It is elementary to check that every contractive multiplier of $\cH$ is also a contractive
  multiplier of $\cH \odot \cH$; alternatively, this also follows from the implication (iii) $\Rightarrow$ (i)
  of Theorem \ref{thm:mult_char}.
  
  Conversely, if $\theta$ is a contractive multiplier of $\cH \odot \cH$,
  then the implication (i) $\Rightarrow$ (iii) of Theorem \ref{thm:mult_char} shows that
  there are $(\varphi_n), (\psi_n) \in M^C_1(\cH)$ so that $\theta = \sum_{n=1}^\infty \varphi_n \psi_n$. In other words, we have
  \begin{equation*}
    \theta =
    \begin{bmatrix}
      \varphi_1 & \varphi_2 & \cdots
    \end{bmatrix}
    \begin{bmatrix}
      \psi_1 \\ \psi_2 \\ \vdots
    \end{bmatrix}.
  \end{equation*}
  The column is a contractive multiplier of $\cH$ as $(\psi_n) \in M^C_1(\cH)$, and the row has norm at most $\kappa$
  by the column-row property since $(\varphi_n) \in M^C_1(\cH)$. Hence $\theta \in \Mult(\cH)$ and
  $\|\theta\|_{\Mult(\cH)} \le \kappa$.
\end{proof}

Matrix-valued multipliers play an important role in the theory of complete Nevan\-linna--Pick spaces, and hence so does the operator space structure of $\Mult(\cH)$. For instance, this operator space structure 
encodes the difference between the Nevanlinna--Pick property and the complete Nevanlinna--Pick property.
In the final section of the paper, we show how to equip $\Mult(\cH \odot \cH)$ with an operator space structure,
i.e.\ we define norms for matrices of multipliers of $\cH \odot \cH$. We establish a version
of Theorem \ref{thm:mult_char} for matrices of multipliers (Theorem \ref{thm:weak_product_vector}) and then use this result
to compare the operator space structures of $\Mult(\cH)$ and of $\Mult(\cH \odot \cH)$.
In the case of the Drury--Arveson space and of the classical Dirichlet space,
we show that while the inclusion $\Mult(\cH) \hookrightarrow \Mult(\cH \odot \cH)$ is a Banach space isomorphism
(by Corollary \ref{cor:column_row}), it is not an isomorphism of operator spaces.

\section{Preliminaries}
\label{sec:prelim}

Let $\cH$ be a normalized complete Nevanlinna--Pick space of functions on a set $X$.
We assume throughout that $\cH$ is separable.

\subsection{Hankel operators}
\label{ss:Hankel}

Nehari's theorem identifies the dual space of $H^1(\bD)$ with the space of
symbols of bounded Hankel operators on $H^2(\bD)$ via the standard integral
pairing on the unit circle.
When studying multipliers of $\cH \odot \cH$, we will heavily use a generalization of
this fact, namely the duality between $\cH \odot \cH$ and the space $\Han(\cH)$ of symbols of Hankel operators
on $\cH$. We now recall the necessary basics from \cite[Section 2]{AHM+18}.

The conjugate Hilbert space of $\cH$ is
\begin{equation*}
  \ol{\cH} = \{ \ol{f}: f \in \cH \}.
\end{equation*}
This space can be \emph{linearly} and isometrically identified with the dual space $\cH^*$ of $\cH$
by means of the inner product of $\cH$. Notice that $\ol{\cH}$ is again a normalized complete Nevanlinna--Pick
space on $X$ whose reproducing kernel is the complex conjugate of that of $\cH$.
The map $f\mapsto \ol{f}$ yields an anti-unitary operator between $\cH$ and $\ol{\cH}$ which conjugates $\Mult(\cH)$ to $\Mult(\ol{\cH})$.

It was shown in \cite[Section 2]{AHM+18} that the dual space of $\cH \odot \cH$ can
be linearly and isometrically identified with a subspace $\HAN(\cH) \subset B(\cH, \ol{\cH})$;
operators in $\HAN(\cH)$ are called Hankel operators on $\cH$. Every operator $T \in \HAN(\cH)$
is uniquely determined by its symbol $T 1 \in \ol{\cH}$. Let
\begin{equation*}
  \Han(\cH) = \{ \ol{T 1} \in \cH: T \in \HAN(\cH) \}
\end{equation*}
denote the space of symbols of Hankel operators.
Given $b \in \Han(\cH)$, we write $H_b$ for the unique operator in $\HAN(\cH)$ that satisfies
$H_b 1 = \ol{b}$. Notice that the map $b \mapsto H_b$ is conjugate linear.
The operator $H_b$ is uniquely determined by the requirement that
\begin{equation}
  \label{eqn:Han_def}
  \langle H_b f, \ol{\varphi} \rangle_{\ol{\cH}} = \langle \varphi f,b \rangle_{\cH}
  \quad \text {for all } f \in \cH, \varphi \in \Mult(\cH).
\end{equation}
(Since $\cH$ is a normalized complete Nevanlinna--Pick space, the kernel functions
are contained in $\Mult(\cH)$ and hence $\Mult(\cH)$ is densely contained in $\cH$.)

\begin{rem}
  If $\cH$ satisfies the column-row property, then the space $\Han(\cH)$
  can be more concretely described as the space of all those $b \in \cH$
  for which the densely defined bilinear form on $\cH \times \cH$,
  \begin{equation*}
    (\varphi ,f) \mapsto \langle \varphi f,b \rangle \quad (\varphi \in \Mult(\cH), f \in \cH),
  \end{equation*}
  is bounded; see \cite[Theorem 2.6]{AHM+18}.
\end{rem}

The version of Nehari's theorem obtained in \cite[Theorem 2.5]{AHM+18} asserts that
there is a conjugate linear isometric isomorphism
$\Han(\cH) \cong (\cH \odot \cH)^*, b \mapsto L_b$, satisfying
\begin{equation}
  \label{eqn:weak_product_Han_dual}
  L_b( \varphi f) = \langle \varphi f,b \rangle \quad \text{ for all } b \in \Han(\cH), f \in \cH, \varphi \in \Mult(\cH).
\end{equation}
Here, the norm on $\Han(\cH)$ is given by $\|b\|_{\Han(\cH)} = \|H_b\|_{B(\cH,\ol{\cH})}$.
We write $[f,H_b] = L_b(f)$ for $f \in \cH \odot \cH$ and $b \in \Han(\cH)$.
Thus, combining \eqref{eqn:Han_def} and \eqref{eqn:weak_product_Han_dual}, we see that
the linear duality between $\HAN(\cH)$ and $\cH \odot \cH$ is given by
\begin{equation}
  \label{eqn:HAN_wp_dual}
  [f g, H_b] = \langle H_b f, \ol{g} \rangle_{\ol{\cH}} \quad (f,g \in \cH, b \in \Han(\cH)).
\end{equation}
This duality also endows $\HAN(\cH)$ with a weak-$*$ topology.
It follows from the construction in \cite[Section 2]{AHM+18}, or one checks directly,
that this weak-$*$ topology agrees with the one inherited from $B(\cH,\ol{\cH})$ as the dual of
the space $T(\ol{\cH},\cH)$ of trace class operators.

\begin{rem}
In this article, we find it convenient to distinguish between a Hankel operator and its symbol,
as the correspondence between them is \emph{conjugate} linear.
Ultimately, we will work with the operators directly and only use Equation \eqref{eqn:HAN_wp_dual},
which in principle could be understood without explicitly mentioning symbol functions.
\end{rem}

Given a bounded linear operator $A: \cH \odot \cH \to \cH \odot \cH$,
let $A^\dagger: \HAN(\cH) \to \HAN(\cH)$ be the Banach space adjoint
of $A$ modulo the linear isometric isomorphism $\HAN(\cH) \cong (\cH \odot \cH)^*$.
Equation \eqref{eqn:HAN_wp_dual} shows that the action of $A^\dagger$ is characterized by
the formula
\begin{equation}
  \label{eqn:dagger}
  \langle A^\dagger (H_b) f, \ol{g} \rangle_{\ol{\cH}} = [ A( f g ), H_b ]
  \quad (f,g \in \cH, b \in \Han(\cH)).
\end{equation}

It follows from a theorem of Hartman that $H^1(\bD)$ can be identified with the dual
of the space of all symbols of compact Hankel operators.
In a similar fashion, the weak product $\cH \odot \cH$ is also a dual space.
More precisely, let
\begin{equation*}
  \HAN_0(\cH) = \HAN(\cH) \cap K(\cH,\ol{\cH})
\end{equation*}
be the space  of all compact Hankel operators on $\cH$.
According to Theorems 2.1 and 2.5 of \cite{AHM+18}, the duality
between $\cH \odot \cH$ and $\HAN(\cH)$ also yields an isometric isomorphism
$\HAN_0(\cH)^* \cong \cH \odot \cH$.

If $k_z \in \cH$ denotes the kernel function associated with $z \in X$, then
boundedness of point evaluations on $\cH \odot \cH$ and the duality
between $\cH \odot \cH$ and $\Han(\cH)$ (see Equation \eqref{eqn:weak_product_Han_dual}) imply that $k_z \in \Han(\cH)$. Moreover,
Equation \eqref{eqn:Han_def} shows that $H_{k_z} f = f(z) \ol{k_z}$ for $f\in \cH$, so that $H_{k_z}$
is a rank-one operator and $H_{k_z} \in \HAN_0(\cH)$.
In particular, point evaluations on $\cH \odot \cH$ are continuous in the weak-$*$ topology
given by the duality $\cH \odot \cH = \HAN_0(\cH)^*$. On bounded subsets of $\cH \odot \cH$,
the weak-$*$ topology agrees with the topology of pointwise
convergence on $X$; see \cite[Corollary 2.2]{AHM+18} and the remarks preceding it.
It also follows from the Hahn--Banach theorem that $\HAN_0(\cH)$ is weak-$*$ dense in $\HAN(\cH)$.

\subsection{Multiplication operators and duality}

To distinguish multiplication operators on $\cH$ and on $\cH \odot \cH$, we will use the following
notation. Given $\theta \in \Mult(\cH \odot \cH)$, the corresponding multiplication operator
on $\cH \odot \cH$ is denoted by
\begin{equation*}
  M_\theta: \cH \odot \cH \to \cH \odot \cH, \quad f \mapsto \theta \cdot f.
\end{equation*}
Given $\varphi \in \Mult(\cH)$, we denote the associated multiplication operator on $\cH$ by
\begin{equation*}
  T_\varphi: \cH \to \cH, \quad f \mapsto \varphi \cdot f.
\end{equation*}
(If $\cH = H^2(\bD)$, then $T_\varphi$ is an analytic Toeplitz operator, which motivates our choice of notation.)

Since every multiplier of $\cH$ is also a multiplier of $\cH\odot \cH$, if $b\in \Han(\cH)$ and $\psi\in\Mult(\cH)$ then $L_b\circ M_\psi\in (\cH\odot \cH)^*$. This element in the dual is verified to correspond to $T_\psi^*b\in \Han(\cH)$. Moreover,
the defining equation of a Hankel operator, Equation \eqref{eqn:Han_def}, easily implies the
intertwining relation
\begin{equation}
  \label{eqn:Hankel_intertwining}
  H_b T_\psi = T_{\ol{\psi}}^* H_b = H_{ T^*_{\psi} b} \quad \text{ for all }
  \psi \in \Mult(\cH), b \in \Han(\cH);
\end{equation}
see also \cite[Lemma 2.3]{AHM+18}. In particular,
$\HAN(\cH)$ is a $\Mult(\ol{\cH})^*-\Mult(\cH)$-bimodule.
The following lemma shows that if $\theta \in \Mult(\cH \odot \cH)$, then
$M_\theta^\dagger$ respects this bimodule structure.

\begin{lem}
  \label{lem:mult_adjoint}
  Let $\theta \in \Mult(\cH \odot \cH)$, let $\varphi \in \Mult(\cH)$ and let $b \in \Han(\cH)$.
  Then:
  \begin{enumerate}[label=\normalfont{(\alph*)}]
    \item $M_\varphi^\dagger(H_b) = H_b T_\varphi = T_{\ol{\varphi}}^* H_b$.
    \item $M_\theta^\dagger( H_b T_\varphi) = M_\theta^\dagger(H_b) T_\varphi = M_\theta^\dagger( T_{\ol{\varphi}}^* H_b) = T_{\ol{\varphi}}^* M_\theta^\dagger(H_b)$.
    \item $M_\theta^\dagger(H_{k_z}) = \theta(z) H_{k_z}$ for all $z \in X$.
  \end{enumerate}
\end{lem}

\begin{proof}
  (a) For $f,g \in \cH$, we find using \eqref{eqn:HAN_wp_dual} and \eqref{eqn:dagger} that
  \begin{equation*}
    \langle M_\varphi^\dagger (H_b) f, \ol{g} \rangle_{\ol{\cH}}
    = [ \varphi f g, H_b]
    = \langle H_b T_\varphi f, \ol{g} \rangle_{\ol{\cH}},
  \end{equation*}
  from which the first half of (a) follows. The second half follows from \eqref{eqn:Hankel_intertwining}.

  (b)
  Dualizing the commutation relation $M_\varphi M_\theta = M_\theta M_\varphi$ and using (a), we
  see that
  \begin{equation*}
    M_\theta^\dagger( H_b T_\varphi) = (M_\theta^\dagger M_\varphi^\dagger) (H_b)
    = (M_\varphi^\dagger M_\theta^\dagger) (H_b) = M_\theta^\dagger (H_b) T_\varphi.
  \end{equation*}
  The remaining parts of (b) follow from this and from \eqref{eqn:Hankel_intertwining}.

  (c) For $z \in X$ we have that $[h,H_{k_z}]=h(z)$ for every $h\in \cH\odot \cH$. Using \eqref{eqn:HAN_wp_dual} and  \eqref{eqn:dagger}, we obtain
  for $f,g \in \cH$ the identity
  \begin{equation*}
    \langle M_\theta^\dagger (H_{k_z}) f, \ol{g} \rangle_{\overline{\mathcal{H}}}
    = [\theta f g, H_{k_z}]
    = \theta(z) [f g, H_{k_z}]
    = \theta(z) \langle H_{k_z} f, \ol{g} \rangle_{\overline{\mathcal{H}}},
  \end{equation*}
  from which (c) follows.
\end{proof}

\subsection{\texorpdfstring{$\cH \odot \cH$}{H dot H} as an operator space}
\label{ss:os}

As mentioned in the introduction, a key device in our analysis of multipliers of $\cH \odot \cH$
is the observation that the weak product carries a natural operator space structure. We therefore
recall the necessary background from the theory of operator spaces. For precise definitions and more information,
the reader is referred to the books \cite{BL04,ER00,Paulsen02,Pisier03}.

Given a vector space $V$, let $M_n(V)$ be the vector space of all $n \times n$ matrices
with entries in $V$.
An (abstract) operator space is a normed space $V$, together with a norm on each $M_n(V)$ satisfying certain axioms. We will not require the precise form of the axioms and thus
simply refer to \cite[Section 2.1]{ER00}.
Perhaps the most important example of an operator space is the space $B(\cH,\cK)$,
where $\cH$ and $\cK$ are Hilbert spaces. In this case, we can identify $M_n(B(\cH,\cK))$
with $B(\cH^n,\cK^n)$, which equips $M_n(B(\cH,\cK))$ with a norm. In the same
vein, any subspace of $B(\cH,\cK)$ becomes an operator space in this way.

If $V,W$ are operator spaces, each linear map $A: V \to W$ induces for each $n\in \bN$ a linear map
$A^{(n)}: M_n(V) \to M_n(W)$ defined by applying $A$ entrywise. The map $A$ is said to be
completely bounded if
\begin{equation*}
  \|A\|_{c b} = \sup_{n \in \bN} \| A^{(n)} \| < \infty,
\end{equation*}
and completely contractive if $\|A\|_{cb} \le 1$.
Similarly,
$A$ is said to be a complete isometry if each $A^{(n)}$ is an isometry.
We write $\CB(V,W)$ for the space of all completely bounded linear maps from $V$ to $W$,
endowed with the cb norm.
It is a well-known phenomenon in operator space
theory that completely bounded maps exhibit much better behavior than maps that are merely bounded.

If $V$ is an abstract operator space, then its dual space $V^*$ carries a natural
operator space structure, called the dual operator space structure.
It is defined via the identification $M_n(V^*) = \CB(V,M_n)$.
The dual operator space structure has the property
that if $A: V \to W$ is a completely bounded map between operator spaces, then the adjoint
$A^*: W^* \to V^*$ is completely bounded with $\|A^*\|_{c b} = \|A\|_{cb}$;
see \cite[Section 3.2]{ER00}.

We now apply these abstract considerations to our setting of weak products. Since
$\HAN_0(\cH) \subset \HAN(\cH) \subset B(\cH,\ol{\cH})$, the spaces
$\HAN_0(\cH)$ and $\HAN(\cH)$ carry a natural operator space structure.
Since $\cH \odot \cH$ is isometrically isomorphic
to $\HAN_0(\cH)^*$, we may endow $\cH \odot \cH$ with the corresponding dual operator space structure.
Taking duals again, the resulting dual
operator space structure on $(\cH \odot \cH)^* \cong \HAN(\cH)$ agrees with the
operator space structure of $\HAN(\cH)$ inherited from $B(\cH,\ol{\cH})$,
as the identification of $K(\cH,\ol{\cH})^{**}$ with $B(\cH,\ol{\cH})$ is a complete isometry;
see Theorem 1.4.11 in \cite{BL04} and the discussion preceding it.
In particular, it follows that $\cH \odot \cH$ is endowed with the unique operator
space structure that makes $\HAN(\cH)$ the operator space dual of $\cH \odot \cH$ with
respect to the given duality; this is also known as the predual operator space structure,
see \cite[Section 3]{LeMerdy95} for further discussion.

In particular, we see that a linear map $A: \cH \odot \cH \to \cH \odot \cH$
is completely contractive if and only if its adjoint $A^\dagger: \HAN(\cH) \to \HAN(\cH)$
is completely contractive. We will frequently use this fact. Indeed, for our purposes
it will be more convenient to study properties of $A$ through $A^\dagger$, because
the latter acts on a concrete space of operators as opposed to the space $\cH \odot \cH$,
whose operator space structure is less explicit.

We will only use the description of the operator space structure on $\cH \odot \cH$
in terms of the duality given above.
Nevertheless, we will provide a concrete description of the norm on $M_n(\cH \odot \cH)$
in Lemma \ref{lem:M_n_weak_product}.

\section{Proof of the main result}
\label{sec:proof}

We continue to assume throughout that $\cH$ is a normalized complete Nevanlinna--Pick space of functions on $X$.

\subsection{Factorization implies complete contractivity}

We first show that (iii) $\Rightarrow$ (i) in Theorem \ref{thm:mult_char},
that is, that every function that can be factored using a pair of elements in $M_1^C(\cH)$ is a contractive
multiplier of $\cH \odot \cH$. As mentioned in the introduction, this result is
known \cite[Theorem 3.1]{AHM+18}, but we include the short argument for the convenience of the reader.

\begin{prop}
  \label{prop:factorization_is_contractive}
  Let $(\varphi_n), (\psi_n) \in M^C_1(\cH)$ and define $\theta = \sum_{n=1}^\infty \varphi_n \psi_n$.
  Then $\theta$ is a contractive
  multiplier of $\cH \odot \cH$.
\end{prop}

\begin{proof}
  Let $h = \sum_{n=1}^\infty f_n g_n \in \cH \odot \cH$ with $\sum_{n=1}^\infty \|f_n\| \|g_n\| < \infty$.
  By continuity of point evaluations on $\cH$ and by the Cauchy--Schwarz inequality, the sums defining $h$ and $\theta$ converge pointwise absolutely on $X$,
  hence
  \begin{equation*}
    \theta h = \sum_{k,n=1}^\infty (\varphi_k f_n) (\psi_k g_n).
  \end{equation*}
  Since $(\varphi_n)_{n=1}^\infty$ and $(\psi_n)_{n=1}^\infty \in M_1^C(\cH)$, we find that
  \begin{equation*}
  \sum_{n,k=1}^\infty \| \varphi_k f_n \| \|\psi_k g_n\| \leq \sum_{n=1}^\infty \left( \sum_{k=1}^\infty\|\varphi_kf_n\|^2\right)^{\frac{1}{2}} \left(\sum_{k=1}^\infty\|\psi_kg_n\|^2 \right)^{\frac{1}{2}}\\
    \le \sum_{n=1}^\infty \|f_n \| \|g_n\|,
  \end{equation*}
  so taking the infimum over all representations $h = \sum_{n=1}^\infty f_n g_n$, it follows that
  $\|\theta h\|_{\cH \odot \cH} \le \|h\|_{\cH \odot \cH}$, so that $\theta$ is a contractive multiplier
  of $\cH \odot \cH$.
\end{proof}

While not logically necessary, we improve the preceding result by showing
that functions that factor as above are actually \emph{completely} contractive multipliers
of $\cH \odot \cH$;  this is the implication (iii) $\Rightarrow$ (ii) of Theorem
\ref{thm:mult_char}.
We provide this proof as it shows how the operator space structure of $\cH \odot \cH$
and the duality between $\cH \odot \cH$ and $\HAN(\cH)$ enter the picture, and it
foreshadows
the dilation theoretic proof of the reverse
implication (ii) $\Rightarrow$ (iii) of Theorem \ref{thm:mult_char}.

\begin{prop}
  \label{prop:factorization_multiplier}
  Let $(\varphi_n), (\psi_n) \in M^C_1(\cH)$ and define $\theta = \sum_{n=1}^\infty \varphi_n \psi_n$.
  Then $\theta$ is a completely contractive
  multiplier of $\cH \odot \cH$.
\end{prop}

\begin{proof}
  Observe that it suffices to show that for each $N \in \bN$, the
  function $\theta_N = \sum_{n=1}^N \varphi_n \psi_n$ is a completely contractive multiplier
  of $\cH \odot \cH$. Indeed, $\theta_N$ converges pointwise to $\theta$. Hence,
  assuming that each $\theta_N$ is a completely contractive multiplier of $\cH \odot \cH$,
  we see that for all $f \in \cH \odot \cH$, the sequence $(\theta_N f)$ converges to $\theta f$ in the weak-$*$ topology
  of $\cH \odot \cH$. Thus, $\theta$ is completely contractive if each $\theta_N$ is.

  Therefore, we may assume that $\theta = \sum_{n=1}^N \varphi_n \psi_n$ for some $N \in \bN$. In particular, $\theta\in \Mult(\cH)$. We will show that, equivalently, the adjoint map $M_\theta^\dagger: \HAN(\cH) \to \HAN(\cH)$ is completely contractive.
  To this end, we apply part (a) of Lemma \ref{lem:mult_adjoint} to conclude that
  \begin{equation*}
    M_\theta^\dagger(H_b) = H_b T_\theta =
    \begin{bmatrix}
      T_{\ol{\psi_1}}^* & \cdots & T_{\ol{\psi_N}}^*
    \end{bmatrix} \big( H_b \oplus \cdots \oplus H_b \big)
    \begin{bmatrix}
      T_{\varphi_1} \\ \vdots \\ T_{\varphi_N}
    \end{bmatrix}
  \end{equation*}
  for every $b \in \Han(\cH)$.
   This formula implies that $M_\theta^\dagger$ is completely contractive
   once we know that the row and the column are contractive,
   which in turn follows from the assumption $(\varphi_n),(\psi_n) \in M^C_1(\cH)$
  (see also the remarks about the conjugate Hilbert space in Subsection \ref{ss:Hankel}).
\end{proof}

\subsection{Contractive multipliers are completely contractive}

The goal of this subsection is to show that
every contractive multiplier of $\cH \odot \cH$ is completely contractive,
that is, we prove the implication (i) $\Rightarrow$ (ii) of Theorem \ref{thm:mult_char}.

The key tool is the following lemma, which uses a recent result
of Jury and Martin \cite{JM18}.
For notational convenience, we regard finite sequences of multipliers
as infinite sequences that are eventually zero.

\begin{lem}
  \label{lem:matricial_norming}
  Let $[A_{i j}] \in  M_n(B(\cH,\ol{\cH}))$. Then
  \begin{equation*}
    \| [A_{i j}] \| = \sup \Big\{ \Big\| \sum_{i,j=1}^n T_{\ol{\psi_i}}^* A_{i j} T_{\varphi_j} \Big\|:
    (\varphi_i)_{i=1}^n, (\psi_i)_{i=1}^n \in M^C_1(\cH) \Big\}.
  \end{equation*}
\end{lem}

\begin{proof}
  If $(\varphi_i),(\psi_i) \in M_1^C(\cH)$, then
  \begin{equation*}
    \Bigl\| \sum_{i,j=1}^n T_{\ol{\psi_i}}^* A_{i j} T_{\varphi_j} \Big\| =
    \Bigg\|
    \begin{bmatrix}
      T_{\ol{\psi_1}}^* & \cdots & T_{\ol{\psi_n}}^*
    \end{bmatrix}
    [A_{i j}]
    \begin{bmatrix}
      T_{\varphi_1} \\ \vdots \\ T_{\varphi_n}
    \end{bmatrix} \Bigg\| \le \big\| [A_{i j}] \big\|,
  \end{equation*}
  as the row and the column are contractions, hence the inequality ``$\ge$'' holds
  in the statement of the lemma.

  To prove the reverse inequality, it suffices to show that for every pair of sequences
  $(f_i)_{i=1}^n, (g_i)_{i=1}^n$ of elements of $\cH$ with $\sum_{i=1}^n \|f_i\|^2 = \sum_{j=1}^n \|g_i\|^2 = 1$,
  there exist $(\varphi_i),(\psi_i) \in M^C_1(\cH)$ so that
  \begin{equation*}
    \Big| \sum_{i,j=1}^n \langle A_{i j} f_j, \ol{g_i} \rangle_{\ol{\cH}} \Big|
    \le 
    \Big\| \sum_{i,j=1}^n T_{\ol{\psi_i}}^* A_{i j} T_{\varphi_j} \Big\|.
  \end{equation*}
  To this end, we apply Theorem 1.1 of \cite{JM18}, which yields $(\varphi_i),(\psi_i) \in M^C_1(\cH)$ and
  $F,G \in \cH$ with $\|F\| \le 1, \|G\| \le 1$ such that $f_i = \varphi_i F, g_i = \psi_i G$ for all $i$.
  Then
  \begin{equation*}
    \Big| \sum_{i,j=1}^n \langle A_{i j} f_j, \ol{g_i} \rangle_{\ol{\cH}} \Big|
    =\Big| \sum_{i,j=1}^n \langle A_{i j} T_{\varphi_j} F, T_{\ol{\psi_i}} \ol{G} \rangle_{\ol{\cH}} \Big|
    \le \Big\| \sum_{i,j=1}^n T_{\ol{\psi_i}}^* A_{i j} T_{\varphi_j} \Big\|,
  \end{equation*}
  as desired.
\end{proof}

\begin{rem}
  In the language of operator bimodules, Lemma \ref{lem:matricial_norming} says
  that the pair $(\Mult(\ol{\cH})^*,\Mult(\cH))$ is \emph{matricially norming} for $B(\cH,\ol{\cH})$,
  and in particular for $\HAN(\cH)$.
  This property is most commonly studied for $C^*$-bimodules, see for example \cite[Section 8]{Paulsen02}.
\end{rem}

Given the matricial norming property of Lemma \ref{lem:matricial_norming},
it is now routine to finish the proof of the implication (i) $\Rightarrow$ (ii) of Theorem
\ref{thm:mult_char}; cf.\ \cite[Proposition 8.6]{Paulsen02}.

\begin{prop}
  \label{prop:mult_cc}
  Every contractive multiplier of $\cH \odot \cH$ is completely contractive.
\end{prop}

\begin{proof}
  Let $\theta \in \Mult(\cH \odot \cH)$ be a contractive multiplier.
  By duality, it suffices to show that the contractive map $M_\theta^\dagger: \HAN(\cH) \to \HAN(\cH)$
  is completely contractive.
  To this end, we use Lemma \ref{lem:matricial_norming} and part (b)
  of Lemma \ref{lem:mult_adjoint} to see that
  for $[H_{i j}] \in M_n(\HAN(\cH))$,
  \begin{align*}
    \| [ M_\theta^\dagger (H_{i j}) ] \|
    = \sup \Big\| \sum_{i,j=1}^n T_{\ol{\psi_i}}^* M_\theta^\dagger (H_{i j})  T_{\varphi_j} \Big\|
    &= \sup \Big\| M_\theta^\dagger \Big(\sum_{i,j=1}^n T_{\ol{\psi_i}}^* H_{i j} T_{\varphi_j} \Big) \Big\| \\
    &\le \sup \Big\| \sum_{i,j=1}^n T_{\ol{\psi_i}}^* H_{i j} T_{\varphi_j} \Big\| = \|[H_{i j}]\|,
  \end{align*}
  where all suprema are taken over $(\varphi_i), (\psi_i) \in M^C_1(\cH)$.
\end{proof}

\subsection{Completely contractive multipliers admit a factorization}

In this subsection, we prove the remaining implication (ii) $\Rightarrow$ (iii) of
Theorem \ref{thm:multipliers_factor},
that is, the factorization for completely contractive multipliers $\theta$ of $\cH \odot \cH$.
To this end, we use dilation theory to obtain a representation for the adjoint
$M_\theta^\dagger$ as in the proof of Proposition \ref{prop:factorization_multiplier}.
We emphasize that it is \emph{complete} contractivity that enables this use of dilation theory.
The first step is the following consequence of the Haagerup--Paulsen--Wittstock
dilation theorem.

\begin{lem}
  \label{lem:HPW}
  Let $A: \HAN(\cH) \to \HAN(\cH)$ be a completely contractive linear map
  that is (weak-$*$,weak-$*$) continuous. Then there exist linear contractions
  $V: \cH \to \cH \otimes \ell^2$ and $W: \ol{\cH} \to \ol{\cH} \otimes \ell^2$ such that
  \begin{equation*}
    A(H_b) = W^* (H_b \otimes I_{\ell^2}) V
  \end{equation*}
  for all $b \in \Han(\cH)$.
\end{lem}

\begin{proof}
  Recall that $\HAN(\cH) \subset B(\cH,\ol{\cH})$ and $\HAN_0(\cH) \subset K(\cH,\ol{\cH})$.
  To be in the more familiar setting
  of spaces of operators on a single Hilbert space,
  we fix a (non-canonical) linear unitary $U: \ol{\cH} \to {\cH}$ and define
  $\widetilde \HAN_0(\cH) = U \HAN_0(\cH) \subset K(\cH)$
  and
  \begin{equation*}
    \widetilde A: \widetilde \HAN_0(\cH) \to B(\cH), \quad \widetilde A(U H_b)
    = U A(H_b).
  \end{equation*}
  Then $\widetilde A$ is completely contractive, so by the Haagerup--Paulsen--Wittstock
  dilation theorem (Theorems 8.2 and 8.4 in \cite{Paulsen02}), there exist a Hilbert
  space $\cF \supset \cH$, a $*$-representation
  $\pi: K(\cH) \to B(\cF)$ and contractions $X,Y: \cH \to \cF$ so that
  \begin{equation*}
    \widetilde A(U H_b) = X^* \pi(U H_b) Y \quad (H_b \in \HAN_0(\cH)).
  \end{equation*}
  Since $\cH$ is separable, $\cF$ can be chosen to be separable as well.
  Every $*$-representation of $K(\cH)$ is unitarily equivalent to a multiple
  of the identity representation, hence there exist contractions $V_0,W_0: \cH \to \cH \otimes \ell^2$
  so that
  \begin{equation*}
    U A(H_b) = \widetilde A(U H_b) = W_0^* (U H_b \otimes I) V_0 \quad (H_b \in \HAN_0(\cH)).
  \end{equation*}
  Thus, defining $V = V_0$ and
  $W = (U^* \otimes I) W_0 U$, we see that
  \begin{equation}
    \label{eqn:dilation}
    A(H_b) = W^* ( H_b \otimes I) V \quad (H_b \in \HAN_0(\cH)).
  \end{equation}
  Recall from Subsection \ref{ss:Hankel} that $\HAN_0(\cH)$
  is weak-$*$ dense in $\HAN(\cH)$ and that the inclusion
  $\HAN(\cH) \subset B(\cH,\ol{\cH})$ is (weak-$*$,weak-$*$) continuous,
  so the (weak-$*$,weak-$*$) continuity of $A$ therefore implies
  that \eqref{eqn:dilation} holds whenever $H_b \in \HAN(\cH)$.
\end{proof}

\begin{rem}
  The use in the previous proof of the somewhat unnatural operator $U: \ol{\cH} \to \cH$
  can be avoided by using ``rectangular'' dilation theory, see for example \cite{FHL16}.
  In this setting, $A$ dilates to a triple representation of the TRO $K(\cH,\ol{\cH})$,
  and every triple representation of $K(\cH,\ol{\cH})$ is unitarily equivalent
  to a multiple of the identity representation.
\end{rem}

We are ready to prove the remaining implication (ii) $\Rightarrow$ (iii) of Theorem \ref{thm:mult_char}.

\begin{thm}
  \label{thm:multipliers_factor}
  Let $\cH$ be a normalized complete Nevanlinna--Pick space and let $\theta$
  be a completely contractive multiplier of $\cH \odot \cH$.
  Then there exist $(\varphi_n), (\psi_n) \in M_1^C(\cH)$ such that
      $\theta = \sum_{n=1}^\infty \varphi_n \psi_n$.
\end{thm}

\begin{proof}
  Since $\theta$ is a completely contractive multiplier of $\cH \odot \cH$,
  the adjoint map $M_\theta^\dagger: \HAN(\cH) \to \HAN(\cH)$ is a (weak-$*$,weak-$*$) continuous complete contraction.
  Hence, the dilation theoretic Lemma \ref{lem:HPW} implies that there
  exist contractions $V: \cH \to \cH \otimes \ell^2$
  and $W: \ol{\cH} \to \ol{\cH} \otimes \ell^2$ such that
  \begin{equation}
    \label{eqn:han_dilation}
    M_\theta^\dagger(H_b) = W^* (H_b \otimes I) V
  \end{equation}
  for all $b \in \Han(\cH)$. We will show that $V$ and $W$ can be replaced with suitable
  multiplication operators, thus obtaining
  a representation as in the proof of Proposition \ref{prop:factorization_multiplier}.

  To this end, let
  \begin{equation*}
    \cM = \Big( \bigcap_{b \in \Han(\cH)} \ker (W^* (H_b \otimes I)) \Big)^\bot \subset \cH \otimes \ell^2.
  \end{equation*}
  Since $H_b T_\varphi$ is a Hankel operator for all $\varphi \in \Mult(\cH)$, we find that $\cM$
is invariant under $T_{\varphi}^* \otimes I$ for all $\varphi \in \Mult(\cH)$.
  Let $X = P_\cM V$.
  Then Equation \eqref{eqn:han_dilation} implies that
  \begin{equation}
    \label{eqn:han_dilation_X}
    M_\theta^\dagger(H_b) = W^* (H_b \otimes I) X
  \end{equation}
  for all $b \in \Han(\cH)$. Thus, by part (b) of Lemma \ref{lem:mult_adjoint},
  we obtain for all $\varphi \in \Mult(\cH)$ and all $b \in \Han(\cH)$
  the identity
  \begin{align*}
    W^* (H_b \otimes I) X T_\varphi = M_\theta^\dagger(H_b) T_\varphi =
    M_\theta^\dagger(H_b T_\varphi)= W^* (H_b \otimes I) (T_\varphi \otimes I) X.
  \end{align*}
  Therefore, for all $\varphi \in \Mult(\cH)$, we find that
  \begin{equation*}
    W^* (H_b \otimes I) [ X T_\varphi - (T_\varphi \otimes I) X] = 0
  \end{equation*}
  for all $b \in \Han(\cH)$, so the definition of $\cM$ shows that
  \begin{equation*}
    X T_\varphi = P_\cM (T_\varphi \otimes I) X
  \end{equation*}
  for all $\varphi \in \Mult(\cH)$. In this setting, the Ball--Trent--Vinnikov commutant lifting
  theorem (see \cite[Theorem 5.1]{BTV01}) implies that there exists a contractive multiplier
  $\Phi \in \Mult(\cH, \cH \otimes \ell^2)$ with $X = P_\cM T_\Phi$, and hence \eqref{eqn:han_dilation_X}
  shows that
  \begin{equation*}
    M_\theta^\dagger(H_b) = W^*(H_b \otimes I) P_\cM T_\Phi = W^* (H_b \otimes I) T_\Phi
  \end{equation*}
  for all $b \in \Han(\cH)$.

  A similar argument, applied to the space
  \begin{equation*}
    \cN = \bigvee_{b \in \Han(\cH)} \ran ( (H_b \otimes I) T_\Phi) \subset \ol{\cH} \otimes \ell^2,
  \end{equation*}
  shows that there exists a contractive multiplier $\Psi \in \Mult(\cH, \cH \otimes \ell^2)$ such that
  \begin{equation*}
    M_\theta^\dagger(H_b) = T_{\ol{\Psi}}^* (H_b \otimes I) T_\Phi
  \end{equation*}
  for all $b \in \Han(\cH)$.

  To finish the proof, we write $\Phi = (\varphi_n)$ and $\Psi = (\psi_n)$ with $(\varphi_n),(\psi_n) \in M^C_1(\cH)$,
  so that
  \begin{align*}
    \langle M_\theta^\dagger(H_{b}) 1, \ol{1} \rangle_{\ol{\cH}}
    = \langle T_{\ol{\Psi}}^* (H_b \otimes I) T_\Phi 1 ,\ol{1} \rangle_{\ol{\cH}}
    = \sum_{n=1}^\infty \langle H_b \varphi_n, \ol{\psi_n} \rangle_{\ol{\cH}}
    = \sum_{n=1}^\infty \langle \varphi_n \psi_n, b \rangle_{\cH}.
  \end{align*}
  Choosing $b = k_z$ and using part (c) of Lemma \ref{lem:mult_adjoint}, we see that
  \[\theta(z) = \sum_{n=1}^\infty \varphi_n(z) \psi_n(z)\] as desired.
\end{proof}

The proof above shows that Theorem \ref{thm:multipliers_factor} can be regarded as a dilation
theorem for the completely bounded bimodule map $M_\theta^\dagger: \HAN(\cH) \to \HAN(\cH)$.
In different settings, dilation theorems for completely bounded bimodule maps were obtained by several
authors, see for instance \cite[Theorem 3.1]{Smith91} and the references given there.

\section{\texorpdfstring{$\Mult(\cH \odot \cH)$}{Mult(H dot H)} as an operator space}
\label{sec:mult_op_space}

In this section, we endow $\Mult(\cH \odot \cH)$ with an operator space structure.
Recall that if $V$ and $W$ are operator spaces, then $\CB(V,W)$ is the space of all completely bounded maps
from $V$ into $W$. This space becomes itself an operator space, via the identification
$M_n(\CB(V,W)) = \CB(V,M_n(W))$; see \cite[Section 3.2]{ER00}.
It follows from Theorem \ref{thm:mult_char} that every multiplier of $\cH \odot \cH$
defines a completely bounded map on $\cH \odot \cH$, so we can regard $\Mult(\cH \odot \cH) \subset \CB(\cH \odot \cH, \cH \odot \cH)$ and we endow $\Mult(\cH \odot \cH)$ with the resulting operator space structure.

\subsection{Factoring elements of \texorpdfstring{$M_n(\Mult(\cH \odot \cH))$}{Mn(Mult(H dot H))}}

First, we establish a generalization of the equivalence of (ii) and (iii) of Theorem \ref{thm:mult_char}
for elements of $M_n(\Mult(\cH \odot \cH))$.

Given $\Phi,\Psi \in \Mult(\cH\otimes \bC^n, \cH \otimes \ell^2)$, say
\begin{equation*}
  \Phi =
  \begin{bmatrix}
    \varphi_{1 1} & \varphi_{1 2} & \ldots & \varphi_{1 n} \\
    \varphi_{2 1} & \varphi_{2 2} & \ldots & \varphi_{2 n} \\
    \vdots & \vdots & \ddots & \vdots
  \end{bmatrix}
  \quad
  \text{ and }
  \quad
  \Psi =
  \begin{bmatrix}
    \psi_{1 1} & \psi_{1 2} & \ldots & \psi_{1 n} \\
    \psi_{2 1} & \psi_{2 2} & \ldots & \psi_{2 n} \\
    \vdots & \vdots & \ddots & \vdots
  \end{bmatrix},
\end{equation*}
let $\Psi^T$ denote the transpose of the matrix $\Psi$ and define an $n \times n$ matrix
$\Psi^T \Phi$ of functions on $X$ by
\begin{equation*}
  (\Psi^T \Phi)_{i j} = \sum_{k=1}^\infty \varphi_{k j} \psi_{k i} \quad (1\leq i,j\leq n).
\end{equation*}
Note that the sum converges pointwise absolutely by the Cauchy--Schwarz inequality.

With this notation, the norm on $M_n(\Mult(\cH \odot \cH))$ can be described as follows.
\begin{thm}
  \label{thm:weak_product_vector}
  Let $\cH$ be a normalized complete Nevanlinna--Pick space on $X$ and let $\Theta$ be an $n \times n$
  matrix of functions on $X$. The following statements are equivalent.
  \begin{enumerate}[label=\normalfont{(\roman*)}]
    \item The matrix $\Theta$ belongs to the closed unit ball of $M_n(\Mult(\cH \odot \cH))$.
    \item There exist $\Phi,\Psi$ in the closed unit ball of $\Mult(\cH\otimes \bC^n, \cH \otimes \ell^2)$ so that $\Theta = \Psi^T \Phi$.
  \end{enumerate}
\end{thm}

The proof is closely modeled after those of Proposition \ref{prop:factorization_multiplier}
and Theorem \ref{thm:multipliers_factor}. To use duality, we require
the following result refining the fact that $\|A^*\|_{cb} = \|A\|_{c b}$
for a completely bounded map $A: V \to W$.
This result is undoubtedly known, but we were not able to find an explicit reference.

\begin{lem}
  \label{lem:adjoint}
  Let $V$ and $W$ be operator spaces. Then, the map
  \begin{equation*}
    \CB(V,W) \to \CB(W^*,V^*), \quad A \mapsto A^*
  \end{equation*}
  is a complete isometry.
\end{lem}

\begin{proof}
  Let
  \begin{equation*}
    [A_{i j}] \in M_n(\CB(V,W)) = \CB(V,M_n(W)).
  \end{equation*}
  An elementary computation shows that the norm of $[A_{i j}^*]$ in $M_n(\CB(W^*,V^*)) = \CB(W^*,M_n(V^*))$
  is at most that of $[A_{i j}]$ in $M_n(\CB(V,W)) = \CB(V,M_n(W))$. Thus, the map $A \mapsto A^*$
  is a complete contraction. Applying this map again, using that $A^{**}$ agrees with $A$
  on $V$ and the fact that the inclusion of an operator space into its bidual is a complete isometry (see
  \cite[Proposition 3.2.1]{ER00}), we conclude that $A \mapsto A^*$
  is a complete isometry.
\end{proof}

\begin{proof}[Proof of Theorem \ref{thm:weak_product_vector}]
  (ii) $\Rightarrow$ (i) As in the proof of Proposition \ref{prop:factorization_multiplier}, an approximation
  argument allows us to assume that $\Phi,\Psi \in \Mult(\cH\otimes \bC^n, \cH\otimes \bC^N)$ for some $N \in \bN$. In particular, $\Theta\in \Mult(\cH\otimes \bC^n)$. To compute
  the norm of
  \begin{equation*}
    \Theta = [\theta_{i j}] \in M_n(\Mult(\cH \odot \cH)) \subset M_n(\CB( \Mult(\cH \odot \cH), \Mult(\cH \odot \cH))),
  \end{equation*}
  we apply Lemma \ref{lem:adjoint} and instead compute the norm of
  \begin{equation*}
    [ M_{\theta_{i j}}^\dagger] \in M_n( \CB( \HAN(\cH), \HAN(\cH)))
    = \CB( \HAN(\cH), M_n(\HAN(\cH))).
  \end{equation*}
  So let $b \in \Han(\cH)$. An application of part (a) of Lemma \ref{lem:mult_adjoint} shows that, using notation as in the discussion
  preceding Theorem \ref{thm:weak_product_vector},
  \begin{equation*}
    [ M_{\theta_{i j}}^\dagger (H_b)]
    = \sum_{r=1}^N [ T_{\ol{\psi_{r i}}}^* H_b T_{\varphi_{r j}} ]
    = T_{\ol{\Psi}}^* ( H_b \otimes I_{\bC^N}) T_\Phi.
  \end{equation*}
  Since $T_{\Phi}$ and $T_{\ol{\Psi}}$ have norm at most $1$, this formula implies that $[M_{\theta_{ij}}^\dagger]$
  is a completely contractive map, so (i) holds.

  (i) $\Rightarrow$ (ii) We merely sketch the main steps, as the proof closely follows that of Theorem \ref{thm:multipliers_factor}. Let
  $\Theta = [\theta_{i j}]$ be an element of the unit ball of $M_n(\Mult(\cH \odot \cH))$.
  Using duality, more precisely Lemma \ref{lem:adjoint}, it follows that the map
  \begin{equation*}
    \HAN(\cH) \to M_n(\HAN(\cH)), \quad H_b \mapsto [ M_{\theta_{i j}}^\dagger (H_b)],
  \end{equation*}
  is completely contractive. With minor changes, the dilation theoretic argument in the proof of Lemma \ref{lem:HPW}
  yields linear contractions $V: \cH\otimes \bC^n \to \cH \otimes \ell^2$ and $W: \ol{\cH}\otimes \bC^n \to \ol{\cH} \otimes \ell^2$ such that
  \begin{equation*}
    [M_{\theta_{ij}}^\dagger (H_b)] = W^* (H_b \otimes I) V \quad (b \in \Han(\cH)).
  \end{equation*}
  As in the proof of Theorem \ref{thm:multipliers_factor}, the commutant lifting theorem allows us
  to replace $V$ and $W$ with multiplication operators. More precisely, there are contractive multipliers $\Phi, \Psi \in \Mult(\cH \otimes \bC^n, \cH \otimes \ell^2)$
  so that
  \begin{equation}
    \label{eqn:wp_vector}
    [M_{\theta_{ij}}^\dagger (H_b)] = T_{\ol{\Psi}}^*  (H_b \otimes I) T_\Phi \quad (b \in \Han(\cH)).
  \end{equation}
  Somewhat more explicitly, to find $\Phi$, define $\cM \subset \cH \otimes \ell^2$ and $X = P_\cM V$ verbatim
  as in the proof of Theorem \ref{thm:multipliers_factor}. The bimodule property of $M_{\theta_{i j}}^\dagger$ (part (b) of Lemma \ref{lem:mult_adjoint})
  implies that $X (T_\varphi \otimes I_{\bC^n}) = P_{\cM} (T_\varphi \otimes I_{\ell^2}) X$ for all $\varphi \in \Mult(\cH)$,
  hence the commutant lifting theorem applies. Finally, testing \eqref{eqn:wp_vector} for $b = k_z$ yields
  $\Theta(z) = \Psi^T(z) \Phi(z)$, so we have found the desired factorization.
\end{proof}

The ideas used to prove Theorems \ref{thm:mult_char} and \ref{thm:weak_product_vector}
also yield a more concrete description of the norm on $M_n(\cH \odot \cH)$.
If $n = 1$ and $h \in \cH \odot \cH$, then
\begin{equation*}
  \|h\|_{\cH \odot \cH} = \inf \Big\{ \| (f_k) \|_{\cH \otimes \ell^2} \|(g_k)\|_{\cH \otimes \ell^2} : h = \sum_{k=1}^\infty f_k g_k \Big\}.
\end{equation*}
Indeed, this follows from the definition of the norm on $\cH \odot \cH$ by trading constant factors between $f_k$ and $g_k$.
This last formula can be generalized.
The column operator space structure on $\cH$ is defined by the identification $\cH = B(\bC,\cH)$,
and the resulting operator space is denoted by $\cH_c$; see
\cite[Section 3.4]{ER00}. Thus, $M_n(\cH_c) = B(\bC^n,\cH^n)$.
We also require matrices with infinitely many rows. Let $M_{\infty,n}(\cH_c)$ be the space of all matrices
with entries in $\cH$ of the form
\begin{equation*}
  f =
  \begin{bmatrix}
    f_{1 1} & f_{1 2} & \cdots & f_{1 n} \\
    f_{2 1} & f_{2 2} & \cdots & f_{2 n} \\
    \vdots & \vdots & \ddots & \vdots
  \end{bmatrix}
\end{equation*}
satisfying $\sum_{i=1}^\infty \|f_{i j} \|^2 < \infty$ for $1 \le j \le n$.
As we did for finite matrices, we regard such a matrix as a bounded
linear operator from $\bC^n$ to $\cH \otimes \ell^2$, and we set
\begin{equation*}
  \|f\|_{M_{\infty,n}(\cH_c)} = \|f\|_{B(\bC^n, \cH \otimes \ell^2)}.
\end{equation*}
Notice that if $n = 1$, then $M_{\infty,1}(\cH_c) = \cH \otimes \ell^2$ with equality of norms.
Given $f,g \in M_{\infty,n}(\cH_c)$, we define as above an $n \times n$ matrix $g^T f$ of functions on $X$ by
\begin{equation*}
  (g^T f)_{i j}  = \sum_{k=1}^\infty f_{k j} g_{k i} \quad (1 \le i, j \le n).
\end{equation*}
By the Cauchy--Schwarz inequality, the sum converges pointwise absolutely on $X$. 

\begin{lem}
  \label{lem:M_n_weak_product}
  The following assertions are equivalent for an $n \times n$ matrix $h$ of functions on $X$.
  \begin{enumerate}[label=\normalfont{(\roman*)}]
    \item The matrix $h$ belongs to the closed unit ball of $M_n(\cH \odot \cH)$.
    \item There exist $f$ and $g$ in the closed unit ball of $M_{\infty,n}(\cH_c)$   so that $h = g^T f$.
  \end{enumerate}
  Thus, if $h \in M_n(\cH \odot \cH)$, then
  \begin{equation*}
    \|h\|_{M_n(\cH \odot \cH)} = \inf \big\{
      \|f\|_{M_{\infty,n}(\cH_c)}
    \|g\|_{M_{\infty,n}(\cH_c)} : h = g^T f \},
  \end{equation*}
  and the infimum is attained.
\end{lem}

\begin{proof}
  (ii) $\Rightarrow$ (i)
  By the Cauchy--Schwarz inequality, the sum defining each entry of $h$ converges
  absolutely in the Banach space $\cH \odot \cH$. In particular, each entry of $h$ belongs to $\cH \odot \cH$.
  We will show that $h$ belongs to the unit ball of $M_n(\cH \odot \cH)$.
  By definition of the operator space structure on $\cH \odot \cH$ as the dual of $\HAN_0(\cH)$,
  we have to show that the map
  \begin{equation*}
    A: \HAN_0(\cH) \to M_n, \quad H_b \mapsto \Big[ [h_{i j}, H_b ] \Big]_{i,j},
  \end{equation*}
  is completely contractive, where
  the inner brackets denote the duality between $\cH \odot \cH$ and $\HAN(\cH)$.
  To this end, notice that for $1 \le i,j \le n$,
  Equation \eqref{eqn:HAN_wp_dual} implies that
  \begin{equation*}
    [h_{i j}, H_b] = \sum_{k=1}^\infty [ f_{k j} g_{k i}, H_b]
    = \sum_{k=1}^\infty \langle H_b f_{k j}, \ol{g_{k i}} \rangle_{\ol{\cH}}
  \end{equation*}
  for all $b \in \Han(\cH)$.
  Let $\ol{g}$ denote the entry-wise complex conjugate of $g$, regarded as a contractive
  operator from $\bC^n$ to $\ol{\cH} \otimes \ell^2$, and let $\ol{g}^*: \ol{\cH} \otimes \ell^2 \to \bC^n$
  be the Hilbert space adjoint of $\ol{g}$. Then
  \begin{equation*}
    A(H_b) = \Big[ [h_{i j}, H_b ] \Big]_{i,j} = \ol{g}^* (H_b \otimes I_{\ell^2}) f \quad (H_b \in \HAN_0(\cH)),
  \end{equation*}
  which implies that the map $A$ is completely contractive.

  (i) $\Rightarrow$ (ii) If $h$ belongs to the unit ball of $M_n(\cH \odot \cH)$,
  then by definition of the operator space structure on $\cH \odot \cH$, the map $A$
  defined in the first part of the proof is completely contractive.
  Applying the Haagerup--Paulsen--Wittstock dilation theorem
  as in the proof of Lemma \ref{lem:HPW} and using the
  fact that every $*$-representation of $K(\cH)$
  is unitarily equivalent to a multiple of the identity representation, we obtain
  linear contractions $V: \bC^n \to \cH \otimes \ell^2$ and $W: \bC^n \to \ol{\cH} \otimes \ell^2$
  so that
  \begin{equation*}
    A(H_b) = W^* (H_b \otimes I_{\ell^2}) V \quad (H_b \in \HAN_0(\cH)).
  \end{equation*}
  Define $f,g \in M_{\infty,n}(\cH)$ by $f = V$
  and $\ol{g} = W$. We see that $f$ and $g$ have norm $1$ and
  \begin{equation*}
    [h_{i j}, H_b] = \sum_{k=1}^\infty \langle H_b f_{k j}, \ol{g_{k i}} \rangle_{\ol{\cH}}
  \end{equation*}
  for $1 \le i,j \le n$ and $H_b \in \HAN_0(\cH)$. Testing this equation for $b = k_z$, we conclude
  that $h(z) = g^T(z) f(z)$, as desired.
\end{proof}

\begin{rem}
Lemma \ref{lem:M_n_weak_product} and Theorem \ref{thm:weak_product_vector} can be restated
  in terms of the Haagerup tensor product $\otimes_h$ of operator spaces (see \cite[Chapter 17]{Paulsen02}, \cite[Paragraph 1.5.4]{BL04} or \cite[Chapter 9]{ER00}), its weak-$*$ version $\otimes_{w*\, h}$ (see \cite[Paragraph 1.6.9]{BL04}) and the
  opposite operator space structure $V^{op}$ of an operator space $V$ (see \cite[Paragraph 1.2.25]{BL04}).
  Concretely, Lemma \ref{lem:M_n_weak_product}
  implies that
  \begin{equation*}
    \cH_c^{op} \otimes_h \cH_c \to \cH \odot \cH, \quad \sum_{n=1}^\infty f_n \otimes g_n \mapsto \sum_{n=1}^\infty f_n g_n,
  \end{equation*}
  is a complete quotient mapping. Theorem \ref{thm:weak_product_vector} implies that
  \begin{equation*}
    \Mult(\cH)^{op} \otimes_{w* \, h} \Mult(\cH) \to \Mult(\cH \odot \cH), \quad \sum_{n=1}^\infty \varphi_n \otimes \psi_n
    \mapsto \sum_{n=1}^\infty \varphi_n \psi_n,
  \end{equation*}
  is a complete quotient mapping. We will not use these formulations.
\end{rem}

\subsection{Comparing the operator space structures of \texorpdfstring{$\Mult(\cH)$ and $\Mult(\cH \odot \cH)$}{Mult(H) and Mult(H dot H)}}

We saw in Corollary \ref{cor:column_row} that if $\cH$ satisfies the column-row property (which is the case for instance
for the Drury--Arveson space), then the inclusion
\begin{equation*}
  \iota: \Mult(\cH) \hookrightarrow \Mult(\cH \odot \cH)
\end{equation*}
is an isomorphism of Banach spaces.
For any normalized complete Nevanlinna--Pick space, the implication (ii) $\Rightarrow$ (i) of Theorem \ref{thm:weak_product_vector} shows that $\iota$ is a complete contraction. If $\cH = H^2(\bD)$, then the norm of $M_n(\Mult(\cH))$
is simply the supremum norm over $\bD$, hence Theorem \ref{thm:weak_product_vector} implies that the same
is true for $M_n(\Mult(\cH \odot \cH))$. In other words, in the case of $H^2(\bD)$, the map $\iota$
is a completely isometric isomorphism. Note however that the entire
space $\CB(\cH \odot \cH)$ is \emph{not} completely boundedly isomorphic
  to an operator algebra unless $\cH \odot \cH$ is isomorphic
to a Hilbert space by \cite[Proposition 5.1.9]{BL04}.

We show that the phenomenon observed above is somewhat special to the univariate Hardy space.

\begin{prop}
  Let $\cH$ be either the Drury--Arveson space $H^2_d$ for $d \ge 2$ or the classical Dirichlet space.
  Then the inclusion
  \begin{equation*}
    \iota: \Mult(\cH) \hookrightarrow \Mult(\cH \odot \cH)
  \end{equation*}
  does not have a completely bounded inverse.
\end{prop}

\begin{proof}
  It follows from Theorem \ref{thm:weak_product_vector} that for each $n \in \bN$, the transpose map
  \begin{equation*}
    M_n(\Mult(\cH \odot \cH)) \to M_n(\Mult(\cH \odot \cH)), \quad \Theta \mapsto \Theta^T,
  \end{equation*}
  is isometric. On the other hand, there exist sequences of multipliers
  $(\varphi_n)$ in $\Mult(\cH)$ that yield a bounded row multiplication operator, but an unbounded column multiplication
  operator. For the Dirichlet space, this can be seen from the discussion preceding Lemma 1 in \cite{Trent04};
  for the Drury--Arveson space, see \cite[Subsection 4.2]{AHM+18}. In particular,
  the norms of the transpose maps
  \begin{equation*}
    M_n(\Mult(\cH)) \to M_n(\Mult(\cH)), \quad \Phi \mapsto \Phi^T,
  \end{equation*}
  are not uniformly bounded in $n$, so that the completely contractive map $\iota$
  does not have a completely bounded inverse.
\end{proof}

In fact, it is possible to determine explicitly the growth of the norms of $(\iota^{-1})^{(n)}$
in the case of the Drury--Arveson space.
We begin with the following easy estimate.

\begin{lem}
  \label{lem:row_column_vector}
  Let $\cH$ be a reproducing kernel Hilbert space that satisfies the column-row property
  with constant $\kappa$. Then
  \begin{equation*}
  \|\Psi^T\|_{\Mult(\cH \otimes \ell^2, \cH\otimes \bC^n)} \le \sqrt{n} \kappa \|\Psi\|_{\Mult(\cH\otimes \bC^n,\cH \otimes \ell^2)}
  \end{equation*}
  for all $\Psi \in \Mult(\cH\otimes \bC^n, \cH \otimes \ell^2)$.
\end{lem}

\begin{proof}
  Suppose that
  \begin{equation*}
  \Psi =
  \begin{bmatrix}
    \psi_{1 1} & \psi_{1 2} & \ldots & \psi_{1 n} \\
    \psi_{2 1} & \psi_{2 2} & \ldots & \psi_{2 n} \\
    \vdots & \vdots & \ddots & \vdots
  \end{bmatrix}
  \end{equation*}
  has multiplier norm at most $1$. Then each of the columns has multiplier norm at most $1$, so
  the column-row property shows that each row
  \begin{equation*}
    R_i =
    \begin{bmatrix}
    T_{\psi_{1 i}} & T_{\psi_{2 i}} & T_{\psi_{3 i}} & \ldots
    \end{bmatrix}
  \end{equation*}
  has norm at most $\kappa$, hence
  \begin{equation*}
    \|\Psi^T\|_{\Mult(\cH \otimes \ell^2, \cH \otimes \bC^n)}^2
    =
    \left\|
    \begin{bmatrix}
      R_1 \\ \vdots \\ R_n
    \end{bmatrix}
    \right\|^2
    = \Big\| \sum_{i=1}^n R_i^* R_i \Big\| \le n \kappa^2. \qedhere
  \end{equation*}
\end{proof}

If $A: V \to W$ is a bounded map between operator spaces, then $\|A^{(n)}\| \le n \|A\|$,
and this inequality is sharp in general; see for instance \cite[Exercise 3.10]{Paulsen02}.
In our setting, the preceding lemma, combined with the implication (i) $\Rightarrow$ (ii) of Theorem \ref{thm:weak_product_vector}, implies the following better upper bound.

\begin{cor}
  \label{cor:cr_cb_inv}
 Let $\cH$ be a normalized complete Nevanlinna--Pick space on $X$ that satisfies the column-row property
  with constant $\kappa$ and let
  \begin{equation*}
    \iota : \Mult(\cH) \to \Mult(\cH \odot \cH)
  \end{equation*}
  be the completely contractive inclusion. Then, $\iota$ is a bijection,
  and
  \begin{equation*}
    \| (\iota^{-1})^{(n)} \| \le \sqrt{n} \kappa
  \end{equation*}
  for all $n \in \bN$. \qed
\end{cor}

In the Drury--Arveson space, the upper bound in the preceding corollary is essentially best possible.
To see this, we require a refinement of the construction in \cite[Subsection 4.2]{AHM+18}.
Given $\{\varphi_1,\ldots,\varphi_n \}\subset \Mult(\cH)$, the quantities
\begin{equation*}
  \|
  \begin{bmatrix}
   \varphi_1 & \varphi_2 & \cdots & \varphi_n
 \end{bmatrix}
  \|_{\Mult(\cH \otimes \bC^n, \cH)}
 \quad \text{ and } \quad
  \left\|
  \begin{bmatrix}
    \varphi_1 \\ \varphi_2 \\ \vdots \\ \varphi_n 
  \end{bmatrix}
  \right\|_{\Mult(\cH, \cH \otimes \bC^n)}
\end{equation*}
are called the row norm and column norm, respectively.

\begin{lem}
  \label{lem:DA_column_row}
  Let $d \ge 2$. Then for all $n \ge 1$, there exists
  $\{\varphi_1,\ldots,\varphi_n \} \subset \Mult(H^2_d)$ with row norm $1$ and column norm $\sqrt{n}$.
\end{lem}

\begin{proof}
  For $0 \le k \le n$, let
  \begin{equation*}
    \psi_k =
    \binom{n}{k}^{\frac{1}{2}}
    z_1^{k} z_2^{n-k}.
  \end{equation*}
 If $\alpha = (\alpha_1,\ldots,\alpha_n) \in \{1,2\}^n$
  is an ordered $n$-tuple, let $z_\alpha = z_{\alpha_1} \ldots z_{\alpha_n}$. Since
  each monomial $z_1^k  z_2^{n-k}$ occurs as one of the monomials
  $z_\alpha$ precisely $\binom{n}{k}$ times, we find that
  \begin{equation*}
    \sum_{k=0}^n T_{\psi_k} T_{\psi_k}^*
    = \sum_{k=0}^n \binom{n}{k} T_{z_1^k z_2^{n-k}} T_{z_1^k z_2^{n-k}}^*
    = \sum_{\alpha\in \{1,2\}^n} T_{z_\alpha} T_{z_\alpha}^*.
  \end{equation*}
  It easily
  follows from the fact that the coordinate functions form a row contraction
  that $\sum_{\alpha} T_{z_\alpha} T_{z_\alpha}^* \le I$, hence the row norm
of $ \{\psi_0,\ldots,\psi_n \}$ is at most $1$.

  On the other hand,
  \begin{align*}
   \left\|\begin{bmatrix} T_{\psi_0} \\ \vdots \\ T_{\psi_n}\end{bmatrix} 1 \right\|^2&=\sum_{k=0}^n \|\psi_k\|^2= \sum_{k=0}^n \binom{n}{k} \|z_1^k z_2^{n-k} \|^2 = n+1.
  \end{align*}
  Hence, the column norm of $\{\psi_0,\ldots,\psi_n\}$ is at least $\sqrt{n+1}$. Since
  the column norm is also dominated by
  \begin{equation*}
    \Big\| \sum_{k=0}^n T_{\psi_k}^* T_{\psi_k} \Big\|^{\frac{1}{2}} \le \sqrt{n+1} \max_{0 \le k \le n} \|T_{\psi_k}\| \le \sqrt{n+1},
  \end{equation*}
  the estimates for both the column and the row norm are in fact equalities.
\end{proof}

Thus, we obtain the exact behavior of $\|(\iota^{-1})^{(n)}\|$, up to multiplicative constants,
in the case of the Drury--Arveson space.

\begin{prop}
  Let $d \ge 2$ and consider the completely contractive inclusion
  \begin{equation*}
    \iota : \Mult(H^2_d) \to \Mult(H^2_d \odot H^2_d).
  \end{equation*}
  Then, there exists a constant $\kappa>0$ depending only on $d$, so that
  \begin{equation*}
    \sqrt{n} \le \| (\iota^{-1})^{(n)} \| \le \kappa \sqrt{n}
  \end{equation*}
  for all $n \ge 1$.
\end{prop}

\begin{proof}
  The upper bound follows from Corollary \ref{cor:cr_cb_inv} and the column-row property for $H^2_d$; see \cite[Theorem 1.5]{AHM+18}.

  To obtain the lower bound,
  we use Lemma \ref{lem:DA_column_row} to find a row multiplier
  \begin{equation*}
    \Psi =
    \begin{bmatrix}
      \psi_1 & \psi_2 & \ldots & \psi_n
    \end{bmatrix}
  \end{equation*}
  of norm $1$ so that $\|\Psi^T\|_{\Mult(\cH, \cH\otimes \bC^n)} = \sqrt{n}$.
  Let
  \begin{equation*}
    \Phi =
    \begin{bmatrix}
      1 & 0 & \ldots & 0
    \end{bmatrix}.
  \end{equation*}
  The implication (ii) $\Rightarrow$ (i) of Theorem \ref{thm:weak_product_vector} shows that
  $\Psi^T \Phi$ belongs to the closed unit ball of $M_n(\Mult(\cH \odot \cH))$. On the other hand,
  \begin{equation*}
    \| \Psi^T \Phi\|_{M_n(\Mult(\cH))} = \| \Psi^T \|_{\Mult(\cH, \cH\otimes \bC^n)} = \sqrt{n}
  \end{equation*}
  thus showing that $\| (\iota^{-1})^{(n)} \| \ge \sqrt{n}$.
\end{proof}

\subsection*{Note added in proof} The recent paper \cite{Hartz20} shows that every complete Nevanlinna-Pick space satisfies the column-row property with constant 1. Hence the additional assumption in Corollary \ref{cor:column_row} is automatically satisfied

\bibliographystyle{amsplain}
\bibliography{wp_literature}

\end{document}